\documentclass[a4paper, 12pt]{article}
\usepackage[a4paper, left=3cm, right=3cm]{geometry}
\usepackage[utf8]{inputenc}
\pdfoutput=1
\usepackage[english]{babel}
\usepackage[T1]{fontenc}
\usepackage{mathtools}
\usepackage{amssymb}
\usepackage{paralist}
\usepackage{tikz}
\usepackage{amsmath}
\usepackage{mathrsfs}
\usepackage{tikz-cd}
\usepackage{amsthm}
\usepackage{lscape}
\usepackage{dynkin-diagrams}
\usepackage{ytableau}
\usepackage{enumitem}
\usepackage{hyperref}
\usepackage{multirow}
\usepackage{csquotes}
\usepackage{orcidlink}
\usetikzlibrary{arrows}
\usepackage{booktabs}
\usepackage{url}

\newtheorem{theorem}{Theorem}[section]
\newtheorem{lemma}[theorem]{Lemma}
\newtheorem{corollary}[theorem]{Corollary}

\newtheorem{proposition}[theorem]{Proposition}

\theoremstyle{definition}
\newtheorem{example}[theorem]{Example}
\newtheorem{remark}[theorem]{Remark}
\newtheorem{definition}[theorem]{Definition}

\numberwithin{subcase}{case}
\newtheorem*{Ack}{Acknowledgments}

\begin{document}
\title{On the Gram determinants of the Specht modules}
\author{Linda Hoyer\orcidlink{0009-0009-3710-6371}\footnote{linda.hoyer@rwth-aachen.de}}
\date{Lehrstuhl f\"ur Algebra und Zahlentheorie, RWTH Aachen University, Germany}
\maketitle
\begin{abstract}
For every partition $\lambda$ of a positive integer $n$, let $S^{\lambda}$ be the corresponding Specht module of the symmetric group $\mathfrak{S}_n$, and let $\det(\lambda)\in \mathbb Z$ denote the Gram determinant of the canonical bilinear form with respect to the standard basis of $S^{\lambda}$. Writing $\det(\lambda)=m \cdot 2^{a_{\lambda}^{(2)}}$ for integers $a_{\lambda}^{(2)}$ and $m$ with $m$ odd, we show that if the dimension of $S^{\lambda}$ is even, then $a_{\lambda}^{(2)}$ is also even. This confirms a conjecture posed by Richard Parker in the special case of the symmetric groups.  
\end{abstract}

\section{Introduction}
Let $n$ be a positive integer and let $\mathfrak{S}_n$ be the symmetric group of $n$ letters. There is a canonical bilinear form on the Specht module $S^{\lambda}$ for a partition $\lambda$ of $n$, arising from the embedding of the Specht module in a permutation representation. In \cite{JamesMurphyFormula}, James and Murphy introduced the determinant $\det(\lambda) \in \mathbb{Z}$ of the Gram matrix of the canonical bilinear form of $S^{\lambda}$ with respect to the standard basis. For any prime $p$, we define the integer $a_{\lambda}^{(p)}$ such that \[ \det(\lambda)=m \cdot p^{a_{\lambda}^{(p)}},\] where $m$ is coprime to $p$. We further denote $f_{\lambda}:=\dim(S^{\lambda})$ and say that $\lambda$ is \textit{even} (resp. \textit{odd}) if and only if $f_{\lambda}$ is.

A major challenge in representation theory is the determination of the dimensions of the irreducible representations $D^{\lambda}$ of $\mathfrak{S}_n$  in positive characteristic $p$ for $\lambda$ a $p$-regular partition of $n$. Some results are known, for instance by \cite[Theorem 7.3.18]{james1981representation} it holds that \[ \dim(D^{\lambda}) \geq f_{\lambda}- a_{\lambda}^{(p)}.\] Now, there is an easy formula for $f_{\lambda}$ given by the hook lengths of the Young diagram $[\lambda]$, leaving open if there is something to be said about $a_{\lambda}^{(p)}$.

In this paper, we will show that for $p=2$, the following statement holds:
\begin{theorem}
\label{Main-Thm-Symmetric}
Let $\lambda$ be an even partition of $n$. Then $a_{\lambda}^{(2)}$ is even.
\end{theorem}

This confirms a special case of a conjecture of Richard Parker (see \cite[Conjecture~1.3]{NebeOrthDisc}) about the Gram determinants (called the \textit{orthogonal determinants} in loc. cit.) of the $\rho(G)$-invariant, non-degenerate, symmetric bilinear forms of the representations $\rho$ affording the \textit{orthogonally stable} characters of $G$, for $G$ an arbitrary finite group.

Note that Theorem \ref{Main-Thm-Symmetric} fails to hold for some odd partitions; indeed, already $\det((1,1))=2$. This implies that understanding the parity of a partition will be of importance to us. 

We now explain our strategy for the proof of Theorem \ref{Main-Thm-Symmetric}.
Let $r$ be the integer such that $2^r \leq n < 2^{r+1}$. In \cite[Theorem A]{HookRemovalNavarro} it was shown that for an odd partition $\mu$ of $n$, there is a unique odd partition $D(\mu)$ of $n-2^r$ such that the character $\chi_{D(\mu)}$ of $\mathfrak{S}_{n-2^r}$ is a constituent of odd multiplicity of $\mathrm{Res}^{\mathfrak{S}_n}_{\mathfrak{S}_{n-2^r}}(\chi_{\mu})$. The partition $D(\mu)$ can explicitly be constructed by removing the unique $2^r$-hook from the Young diagram of $\lambda$, see also \cite{OddLattice}. We extend the definition of $D$ to the set of all partitions, where $D$ in particular preserves the parity. We define the \textit{oddness rank} to be the minimal number of iterations of $D$ until the result becomes stationary, with a partition having maximal oddness rank if and only if it is odd.

The formula given in \cite{JamesMurphyFormula} for $\det(\lambda)$ uses certain partitions that arise by removing the rim of some cell $(i,j)$ from the Young diagram of $\lambda$ and then wrapping it around some row $k$ with $k<i$. In \cite{Branches}, these partitions were called the \textit{branches} of $\lambda$. We will investigate the \textit{odd branches} $\mathscr{P}_{\lambda}^{\mathrm{odd}}$, for which we introduce an equivalence relation on $\mathscr{P}_{\lambda}^{\mathrm{odd}}$ by saying that $\mu_1 \sim \mu_2$ if and only if $D(\mu_1)=D(\mu_2)$. The desired statement then follows by an induction on the oddness rank of $\lambda$ and the understanding of the arising equivalence classes of $\mathscr{P}_{\lambda}^{\mathrm{odd}}$.

This paper is a slightly modified version of results that have appeared in the author's PhD thesis \cite{HoyerThesis}. In this thesis, we investigated the orthogonal determinants of the finite groups of Lie type, the finite Coxeter groups and the alternating groups. As it turns out, Theorem \ref{Main-Thm-Symmetric} has far-reaching consequences for the orthogonal determinants of other classes of groups: Since the characters of the alternating groups and the characters of the Coxeter groups of types $B_n$ and $D_n$ are built from the characters of $\mathfrak{S}_n$, the theorem allows us to extend the result to these groups. With some computer help for the exceptional groups, all finite Coxeter groups can be covered. Moreover, Theorem \ref{Main-Thm-Symmetric} is essential in proving \cite[Conjecture 1.3]{NebeOrthDisc} for the  general linear groups $\mathrm{GL}_n(q)$ for an odd prime power $q$; this relies on first reducing to the unipotent characters and then using the formulas of the Gram determinants of the Specht modules of Iwahori--Hecke algebras of type $A_n$ found in \cite{DeterminantsIwahoriHecke}, which will be discussed elsewhere.

    \section{Preliminaries}

 We fix a positive integer $n$. A \textit{partition} is a sequence $\lambda=(a_1, \dots, a_m)$ of non-negative integers such that $a_1 \geq a_2 \geq \dots \geq a_m$. We identify partitions that only differ by a string of zeros at the end. We say that $\lambda$ is a partition of $n$ if $|\lambda|:= a_1+a_2+ \dots +a_m=n.$ We denote by $\mathscr{P}_n$ the set of partitions of $n$ and by $\mathscr{P}:= \bigcup_{i=0}^{\infty} \mathscr{P}_i$ the set of all partitions. If $\mu=(c_1, \dots, c_l)$ is another partition of $n$, we write $\lambda \trianglelefteq \mu$ and say that $\lambda$ proceeds $\mu$ in the \textit{dominance (partial) order} if and only if $a_1 + \dots +a_i \leq c_1+ \dots +c_i$ for all $i \geq 1$. We denote \[[\lambda]=\{ (i,j) \mid (i,j) \in \mathbb Z \times \mathbb Z, 1 \leq i \leq m, 1 \leq j \leq a_m \}\] to be the \textit{Young diagram} of $\lambda$, which we can depict as a set of cells with $a_1$ cells in the first row, $a_2$ cells in the second row, and so on, see Figure \ref{Figure1}. For a cell $c \in [\lambda]$, we denote the \textit{hook length} $h_{\lambda}(c)$ to be the number of cells in the \textit{hook} of $c$, i.e., the number of cells below and to the right of $c$, including itself. The \textit{hook diagram} of $\lambda$ is the Young diagram, where each cell is filled with its hook length.

 \begin{figure}[ht]
     \centering
   \begin{ytableau}
                9 & *(lightgray) 8 & *(lightgray) 7 & *(lightgray) 5 & *(lightgray) 3 & *(lightgray) 2 & *(lightgray) 1 \\
                5 & *(lightgray) 4 &  3 & 1 \\
                3 & *(lightgray) 2 &  1
            \end{ytableau}
     \caption{Hook diagram of $(7,4,3)$, with the hook of the cell $c=(1,2)$ colored gray.}
     \label{Figure1}
 \end{figure}

Note that the entries of the first column of a hook diagram uniquely determine the underlying partition. Generalizing that observation, a \textit{sequence of $\beta$-numbers} is a sequence $\beta=(b_1, \dots, b_m)$ of non-negative integers such that all its entries are distinct. For each such sequence, there is a unique $\sigma \in \mathfrak{S}_m$ such that $\beta=(b_{\sigma(1)}, \dots, b_{\sigma(m)})$ is a strictly decreasing sequence, i.e., $b_{\sigma(1)} > b_{\sigma(2)} > \dots > b_{\sigma(m)}$. We define the partition \[\mathrm{Part}(\beta) \coloneqq (b_{\sigma(1)}-m+1,b_{\sigma(2)}-m+2,\dots,b_{\sigma(m)}-m+m)\] and say that $\beta$ is a sequence of $\beta$-numbers for $\mathrm{Part}(\beta)$. Moreover, we say that $\ell(\beta) \coloneqq m$ is the length of $\beta$ and define $|\beta|:= |\mathrm{Part}(\beta)|$.  

Let $\mathscr{B}$ be the set of all sequences of $\beta$-numbers. We define the subsets \[\mathscr{B}_n:= \{ \beta \in \mathscr{B} \mid |\beta|=n \}\] of the sequences of $\beta$-numbers of $n$ and \[\mathscr{B}(\lambda):= \{ \beta \in \mathscr{B} \mid \mathrm{Part}(\beta)=\lambda \}\] of the sequences of $\beta$-numbers for the partition $\lambda$.

\begin{example}
Let $\lambda=(7,4,3)$ as in Figure \ref{Figure1}. The first column of its hook diagram consists of the sequence $(9,5,3)$, so $(9,5,3)$ is a sequence of $\beta$-numbers for $\lambda$. Further sequences of $\beta$-numbers for $\lambda$ include $(10,6,4,0)$, $(11,7,5,1,0)$ and $(1,0,6,8,12,2)$.  
\end{example}
 

Recall that there is a 1-to-1 correspondence between the partitions $\lambda$ of a positive integer $n$ and the irreducible characters $\chi_{\lambda}$ of the symmetric group $\mathfrak{S}_n$. The following well-known formula gives a convenient way to calculate the degree of each character.  
\begin{proposition}(cf. \cite[Theorem 2.3.21]{james1981representation})
\label{DegreeSpechtModule}
Let $\lambda$ be a partition of $n$. Then 
\[
f_{\lambda}:=\chi_{\lambda}(1)=\frac{n!}{\prod_{c \in [\lambda]} h_{\lambda}(c)}.
            \]
We further set $f_{(0)}:=1$. A partition is said to be odd (resp. even) if and only if $f_{\lambda}$ is odd (resp. even).           
\end{proposition}

We say that a partition $\lambda$ is \textit{$q$-core} for some positive integer $q$ if there is no $c \in [\lambda]$ such that $q \mid h_{\lambda}(c)$. There is a unique partition $\mathrm{core}_q(\lambda)$, which is $q$-core, that can be constructed from $[\lambda]$ by successively removing hooks of cells with hook lengths a multiple of $q$, see also \cite[Section 2.7]{james1981representation}. This is where sequences of $\beta$-numbers come into play; removing a hook from a Young diagram corresponds to subtracting a multiple of a unit vector from a sequence of $\beta$-numbers for $\lambda$.

\begin{example}
 Let $\lambda=(7,4,3)$. Then $\lambda$ is not $8$-core, as $h_{\lambda}((1,2))=8$. Removing the corresponding hook, we end up with $\mathrm{core}_8(\lambda)=(3,2,1)$.    
\end{example}

We shall need a criterion on the parity of $f_{\lambda}$. Let $r$ be the non-negative integer such that $2^r \leq n <2^{r+1}$.

\begin{lemma}(cf. \cite[Lemma 1]{OddLattice})
\label{OddnessPartitions}
Let $\lambda$ be a partition of $n$. There exists at most one cell $c \in [\lambda]$ such that $h_{\lambda}(c)=2^r$. Moreover, $\lambda$ is odd if and only if such a cell exists and $\mathrm{core}_{2^r}(\lambda)$ is odd. 
\end{lemma}

Let $e_i=(0,0, \dots,0,1,0,\dots)$ be the sequence that has a $1$ in position $i$ and $0$ everywhere else.
\begin{corollary}
\label{OddnessBetaNumbers}
Let $\lambda$ be a partition of $n$ and let $\beta_{\lambda} \in \mathscr{B}(\lambda)$. Then there is at most one index $i$ such that $\beta_{\lambda}-2^r e_i \in \mathscr{B}_{n-2^r}$. If such an index, say $j$, exists, then $\beta_{\lambda}-2^r e_j \in \mathscr{B}(\mathrm{core}_{2^r}(\lambda))$.  
\end{corollary}

\begin{definition}
We define the function $D: \mathscr{P} \to \mathscr{P}$ such that for any partition $\lambda$ of $n$, we have $D(\lambda):=\mathrm{core}_{2^r}(\lambda)$. Moreover, we set $D((0)):=(0)$. Denote $D^0(\lambda):=\lambda$ and $D^{m}(\lambda):=D(D^{m-1}(\lambda))$ for any positive integer $m$. We define the \textit{oddness rank} of the partition $\lambda$ by
\[ \mathrm{OddRank}(\lambda) = \min \{m \in \mathbb Z_{\geq 0} \mid D^{m+1}(\lambda)=D^m(\lambda)\}.\] 
\end{definition}

\begin{corollary}
Let $\lambda$ be a partition of $n$. Let $n= \sum_{i=0}^r a_i 2^i$ for $a_i \in \{0,1 \}$ be the decomposition of $n$ in binary. Then $\lambda$ is odd if and only if $\mathrm{OddRank}(\lambda)= \sum_{i=0}^r a_i.$
\end{corollary}

We now define a distance function for the set of partitions, where the distance between two partitions is the minimum number of hooks that must be removed and added to transform one Young diagram into the other.

\begin{definition}
\begin{enumerate}[label=(\roman*)]
\item Let $m$ be a positive integer and let $\beta=(b_1,\dots, b_m), \beta'= (b_1',\dots, b_m')\in \mathbb Z^m$. Let $d_m: \mathbb Z^m \times \mathbb Z ^m \to \mathbb Z_{ \geq 0}$ be defined by
\[ 
d_m(\beta,\beta')= \# \{i \mid 1 \leq i \leq m, \ b_i \neq b_j' \ \text{for all} \ 1 \leq j \leq m\},
\]
i.e., $d_m(\beta,\beta')$ is the number of mismatches of $\beta$ and $\beta'$.
\item Let $\lambda, \mu$ be partitions. We choose the positive integer $m$ large enough such that there are sequences of $\beta$-numbers $\beta_{\lambda} \in \mathscr{B}(\lambda)$, $\beta_{\mu} \in \mathscr{B}(\mu)$ such that $\ell(\beta_{\lambda})=\ell(\beta_{\mu})=m$. We define the distance $d:\mathscr{P} \times \mathscr{P} \to \mathbb Z_{ \geq 0}$ by \[ d(\lambda, \mu):=d_m(\beta_{\lambda},\beta_{\mu}).\] 
\end{enumerate}
\end{definition}

It is clear that the distance between two partitions is well-defined and that $(\mathscr{P},d)$ is a metric space. Note that for instance if $\lambda$ is a partition of $n$ with $\mathrm{OddRank}(\lambda)>0$, then $d(\lambda, D(\lambda))=1$.

The following definition covers the \textit{odd branches} of a partition, which are the main objects we will work with in this paper.
\begin{definition}
Let $\lambda$ be a partition of $n$.
\begin{enumerate}[label=(\roman*)]
\item
Let $\mu$ be another partition of $n$. We say that $\mu$ is a branch of $\lambda$ if $\lambda \trianglelefteq \mu$ and $d(\lambda, \mu)=2$. Let $\mathscr{P}_{\lambda}$ be the set of branches of $\lambda$.
\item Let $\mu$ be a branch of $\lambda$ and let $\beta_{\lambda} \in \mathscr{B}(\lambda)$ (resp. $\beta_{\mu} \in \mathscr{B}(\mu)$) with 
\[
\beta_{\lambda}=(b_1,b_2,b_3, \dots, b_m), \
\beta_{\mu}=(b_1',b_2',b_3, \dots, b_m),
\]
such that $b_1>b_2$, $b_1'>b_2'$. We define
\[
h_{\lambda}^1(\mu):=b_1-b_2'=b_1'-b_2, \
h_{\lambda}^2(\mu):=b_2-b_2'=b_1'-b_1.
\]
Note that there are two cells $c_1=(k,j), c_2 =(l,j) \in [\lambda]$ with $k<l$ such that $h_{\lambda}(c_i)=h_{\lambda}^i(\mu)$ for $i=1,2$. Also, $(b_1,b_2',b_3, \dots, b_m) \in \mathscr{B}$ and the associated partition corresponds to removing the hook of $c_2$ from the Young diagram of $\lambda$.
\item We define \[ \mathscr{P}_{\lambda}^{\mathrm{odd}}:=   \{ \mu \in \mathscr{P}_{\lambda} \mid  \mu \ \textit{is odd}  \} \] to be the set of odd branches of $\lambda$. 
\end{enumerate}
\end{definition}

We are now ready to give the formula for the Gram determinant of a Specht module:
\begin{theorem}(cf. \cite{JamesMurphyFormula})
\label{James-Murphy Formula}
Let $\lambda$ be a partition of $n$. The Gram determinant of the canonical bilinear form with respect to the standard basis of the Specht module $S^{\lambda}$ is given by
\[
\det(\lambda) :=\prod_{\mu \in \mathscr{P}_{\lambda}} \left( \frac{h_{\lambda}^1(\mu) }{h_{\lambda}^2(\mu)} \right)^{\varepsilon_{\mu} f_{\mu}} \in \mathbb{Z},
\]
where $\varepsilon_{\mu} \in \{-1,1 \}$. For any prime $p$, we define the integer $a_{\lambda}^{(p)}$ such that \[ \det(\lambda)=m \cdot p^{a_{\lambda}^{(p)}},\] where $m$ is coprime to $p$.
\end{theorem}

\begin{remark}
\label{RemarkDeterminantFormula}
 We denote $(\mathbb{Q}^{\times})^2:= \{ x^2 \mid x \in \mathbb Q^{\times}  \}$ to be the set of invertible rational squares. Let $\lambda$ be an even partition of $n$ and let $\chi_{\lambda}$ be the character corresponding to $\lambda$. The orthogonal determinant $\det(\chi_{\lambda}) \in \mathbb{Q}^{\times}/(\mathbb{Q}^{\times})^2$ in the sense of \cite{NebeOrtDet} is defined to be the rational square class of the Gram determinant of any $\rho(\mathfrak{S}_n)$-invariant, non-degenerate, symmetric bilinear form of any rational representation $\rho$ affording the character $\chi_{\lambda}$. The obvious choice is the canonical bilinear form for the Specht module $S^{\lambda}$, and we arrive at
  \[
   \det(\chi_{\lambda})= \det(\lambda) \cdot (\mathbb{Q}^{\times})^2= \prod_{\mu \in \mathscr{P}_{\lambda}^{\mathrm{odd}}} \frac{h_{\lambda}^1(\mu) }{h_{\lambda}^2(\mu)}  \cdot (\mathbb{Q}^{\times})^2.
   \]
\end{remark}

\begin{example}
\label{Example21111331}
Let $\lambda=(2,1^{(5)})$ and let $\mu=(3,3,1)$ be two partitions of $7$. Clearly, $\lambda \trianglelefteq \mu$ in the dominance order.

We choose $\beta_{\lambda}:=(7,5,4,3,2,1)$ and $\beta_{\mu}=(8,7,4,2,1,0)$ for the sequences of $\beta$-numbers for $\lambda$ and $\mu$, respectively. Comparing the two sequences, we see that there are two mismatches, so $d(\lambda, \mu)=d_6(\beta_{\lambda}, \beta_{\mu})=2$. So $\mu$ is in fact a branch of $\lambda$; we calculate that $h_{\lambda}^1(\mu)=5$ and $h_{\lambda}^2(\mu)=3$.

It is easily seen that $\mu$ is odd; indeed, $D(\mu)=(3)$ and $f_{(3)}=1$, so the oddness follows by Lemma \ref{OddnessPartitions}. In particular, $\mu$ is an odd branch of $\lambda$. With the same reasoning we see that $\lambda$ is even, since $D(\lambda)=(2,1)$ and $(2,1)$ is $2$-core.

We further observe that $\mathscr{P}_{\lambda}^{\mathrm{odd}}= \{\mu, (7), (5,1,1), (3,1^{(4)}) \}$ and thus arrive at \[\det(\chi_{\lambda})= \frac{5}{3} \cdot \frac{7}{5} \cdot \frac{7}{3} \cdot \frac{7}{1} \cdot (\mathbb{Q}^{\times})^2= 7 \cdot (\mathbb{Q}^{\times})^2.\]

\end{example}

 \begin{figure}[ht]
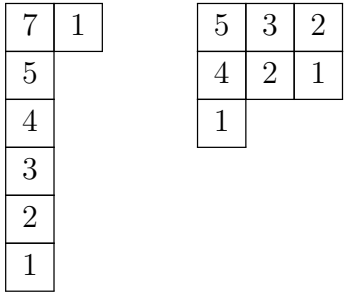

     \centering
   \begin{ytableau}
               7 & 1 \\
               5 \\
               4 \\
               3 \\
               2 \\
               1
            \end{ytableau} \ \ \ \  \ \ \ \ \begin{ytableau}
               5 & 3 & 2 \\
               4 & 2 & 1 \\
               1
            \end{ytableau}
     \caption{Hook diagrams of $(2,1^{(5)})$ and $(3,3,1)$.}
       \label{Figure2}
 \end{figure}

\section{Proof of Theorem \ref{Main-Thm-Symmetric}}
\label{Symmetric Parker}
We fix a positive integer $n$ and an even partition $\lambda$. Let $r$ be the integer such that $2^r \leq n < 2^{r+1}$. Our goal is to work towards a proof of Theorem \ref{Main-Thm-Symmetric}. The following definition will be useful:

\begin{definition}
Let $M \subseteq \mathscr{P}_{\lambda}^{\mathrm{odd}}$ be any subset. We say that $M$ is odd if there is a square-free odd integer $d$ such that \[ \prod_{\mu \in  M}  \frac{h_{\lambda}^1(\mu) }{h_{\lambda}^2(\mu)} \cdot  (\mathbb{Q}^{\times})^2=d \cdot  (\mathbb{Q}^{\times})^2. \]
\end{definition}

\begin{remark}
\label{RemarkEquivalentThm}
Recall that $\det(\lambda)=m \cdot 2^{a_{\lambda}^{(2)}}$ for some odd integer $m$. Now, the integer $a_{\lambda}^{(2)}$ being even is equivalent to $2^{a_{\lambda}^{(2)}}$ being a square. Together with Remark \ref{RemarkDeterminantFormula}, we conclude that Theorem \ref{Main-Thm-Symmetric} holds true if it can be shown that $\mathscr{P}_{\lambda}^{\mathrm{odd}}$ is odd, for any even partition $\lambda$. 
\end{remark}

The idea now is to find a suitable set partition \[\mathscr{P}_{\lambda}^{\mathrm{odd}}=M_1 \cup M_2 \cup \dots \cup M_k \] such that each of the $M_i$ is odd. 
\begin{definition}
We set an equivalence relation on the set of odd branches of $\lambda$ by saying that $\mu_1 \sim \mu_2$ for partitions $\mu_1, \mu_2 \in \mathscr{P}_{\lambda}^{\mathrm{odd}}$ if and only if $D(\mu_1)=D(\mu_2)$. For any odd branch $\mu$ of $\lambda$ we denote with $(\mu)_{\sim}$ the equivalence class of $\mu$. We may also write $D((\mu)_{\sim}):=D(\mu)$.
\end{definition}

\begin{example}
\label{Example211111EquivalenceClasses}
Let again $\lambda=(2,1^{(5)})$ as in Example \ref{Example21111331}. Recall that are four odd branches of $\lambda$. We calculate that \[ D((3,3,1))=D((7))=D((3,1^{(4)}))=(3) \] and $D((5,1,1))=(1,1,1)$, so $\mathscr{P}_{\lambda}^{\mathrm{odd}}$ decomposes into two equivalence classes.
\end{example}

By definition, for any odd branch $\mu$ of $\lambda$, we have $d(\lambda, \mu)=2$. Since $d$ is a distance function, it follows that $d(\lambda, D(\mu)) \in \{1,2,3 \}$. We will first observe that the case $d(\lambda, D(\mu))=1$ cannot occur.

\begin{lemma}
\label{Distance 1}
Let $\mu$ be an odd branch of $\lambda$. Then $d(\lambda, D(\mu)) \in \{ 2,3 \}$. Moreover, $2^r \not\in \{h^1_{\lambda}(\mu), h^2_{\lambda}(\mu)   \}$.
\end{lemma}
\begin{proof}
Assume for a contradiction that  $d(\lambda, D(\mu))=1$. Choose sequences of $\beta$-numbers $\beta_{\lambda}:=(b_1,b_2, \dots, b_m)$ and $\beta_{D(\mu)}:=(b_1',b_2, \dots, b_m)$ for $\lambda$ and $D(\mu)$, respectively. Since $\lambda$ is a partition of $n$ and $D(\mu)$ is a partition of $n-2^r$, it is clear that $b_1'=b_1-2^r$. By Corollary \ref{OddnessBetaNumbers} it now follows that $D(\lambda)=D(\mu)$. But now by Lemma \ref{OddnessPartitions} the parities of $f_{\lambda}$ and $f_{\mu}$ have to be equal, which is a contradiction.

For the second statement we will also be using contradiction, so let $\beta_{\mu}:=(b_1', b_2', b_3, \dots, b_m)$ be a sequence of $\beta$-numbers for $\mu$ and assume that either $h^1_{\lambda}(\mu)$ or $h^2_{\lambda}(\mu)$ is equal to $2^r$. This implies, after maybe a renumbering, that $b_1'=b_1+2^r$ and $b_2'=b_2-2^r$. Then both $\beta_{\lambda}-2^r e_2$ and $\beta_{\mu}-2^r e_1$ are sequences of $\beta$-numbers for $D(\mu)=D(\lambda)$, which is again a contradiction to Lemma \ref{OddnessPartitions}.
\end{proof}

There are thus two cases remaining for an odd branch $\mu$ of $\lambda$; either $d(\lambda, D(\mu))$ is equal to $2$ or $3$. The case of the distance being $3$ is easily understood:
\begin{lemma}
\label{Distance 3}
Let $\mu$ be an odd branch of $\lambda$ with $d(\lambda, D(\mu))=3$. Then $|(\mu)_{\sim}|=1$, $\mathrm{OddRank}(\lambda)>0$ and $D(\mu)$ is an odd branch of  $D(\lambda)$. Moreover, $h^{i}_{\lambda}(\mu)=h^{i}_{D(\lambda)}(D(\mu))$ for $i=1,2$.
\end{lemma}
\begin{proof}
Let $\beta_{\lambda}:=(b_1, b_2, b_3, \dots, b_m) \in \mathscr{B}(\lambda)$. After reordering we can assume that there are  $h_1, h_2, h_3 \in \mathbb Z \backslash \{ 0 \}$ such that \[\beta_{D(\mu)}:= (b_1-h_1, b_2-h_2, b_3-h_3, \dots, b_m) \in \mathscr{B}(D(\mu)).\] For any partition $\kappa \in (\mu)_{\sim}$, there has to be an index $j$ such that $\beta_{D(\mu)}+2^r e_{j} \in \mathscr{B}(\kappa).$ Since $d(\lambda, \kappa)=2$, it follows that $j \in \{ 1, 2, 3 \}$ and that $h_j=2^r$, assume without loss of generality that $j=3$. By Lemma \ref{Distance 1} we know that $2^r \not\in \{ |h_1|, |h_2| \}$, so there can only be one such index and $\kappa=\mu$.

Since $d(\lambda, D(\mu))=3$, the entry $b_3-h_3=b_3-2^r$ cannot appear in $\beta_{\lambda}$. So  
\[\beta_{\lambda}-2^r e_3= (b_1,b_2,b_3-2^r, \dots, b_m) \in \mathscr{B}_{n-2^r}.\] By Corollary \ref{OddnessBetaNumbers} it follows that $\beta_{\lambda}-2^r e_3 \in \mathscr{B}(D(\lambda))$.

The final statement is easy to see: For the calculation of $h^{i}_{\lambda}(\mu)$ for $i=1,2$, we compare the sequences
\begin{align*}
(b_1,b_2,& b_3, \dots, b_m), \\
(b_1-h_1,b_2-h_2,& b_3, \dots, b_m).
\end{align*}
For the calculation of $h^{i}_{D(\lambda)}(D(\mu))$ for $i=1,2$, we compare the sequences
\begin{align*}
(b_1,b_2,& b_3-2^r, \dots, b_m), \\
(b_1-h_1,b_2-h_2,& b_3-2^r, \dots, b_m).
\end{align*}
So the numbers clearly coincide. This concludes the proof.
\end{proof}

\begin{example}
Let $\lambda=(6,3,2)$ and $\mu=(6,4,1)$ be an odd branch of $\lambda$. Then $d(\lambda, D(\mu))=3$ with $D(\lambda)=(2,1)$ and $D(\mu)=(3)$. Moreover, we calculate that $h^1_{\lambda}(\mu)=h^1_{D(\lambda)}(D(\mu))=3$ and $h^2_{\lambda}(\mu)=h^2_{D(\lambda)}(D(\mu))=1$.
\end{example}

\begin{figure}[ht]
     \centering
   \begin{ytableau}
               *(lightgray) 8 & *(lightgray) 7 & *(lightgray) 5 & *(lightgray) 3 & *(lightgray) 2 & *(lightgray) 1 \\
               *(lightgray) 4 & 3 & 1 \\
              *(lightgray)  2 & 1
            \end{ytableau} \  \ \ \ \  \ \ \ \ \begin{ytableau}
               *(lightgray) 8 & *(lightgray) 6 & *(lightgray) 5 & *(lightgray) 4 & *(lightgray) 2 & *(lightgray) 1 \\
               *(lightgray) 5 & 3 & 2 & 1 \\
              *(lightgray)  1 
            \end{ytableau}
     \caption{Hook diagrams of $(6,3,2)$ and $(6,4,1)$.}
       \label{Figure3}
 \end{figure}

We are left with the case that $d(\lambda, D(\mu))=2$ for an odd branch $\mu$ of $\lambda$. This case requires some work, as we have to go over the possible ways the signs of $2^r-h^i_{\lambda}(\mu)$ for $i=1,2$ can distribute. First we gather some general statements.

\begin{lemma}
\label{Distance 2}
Let $\mu$ be an odd branch of $\lambda$ with $d(\lambda, D(\mu))=2$. The following hold:
\begin{enumerate}[label=(\roman*)]
  \item $|(\mu)_{\sim}| \leq 4$.
  \item There is a partition $\kappa \in (\mu)_{\sim} $, which we call a distinguished representative of $(\mu)_{\sim}$, such that the following holds: Let \begin{align*}
\beta_{\lambda}:=&(b_1,b_2, b_3, \dots, b_m) \in \mathscr{B}(\lambda), \\
\beta_{\kappa}:=&(b_1',b_2', b_3, \dots, b_m) \in \mathscr{B}(\kappa)
\end{align*}
 such that $b_1>b_2$ and $b_1'>b_2'$.
  
Then \[\beta_{\kappa}-2^r e_1 = (b_1'-2^r,b_2', b_3, \dots, b_m) \in \mathscr{B}(D(\mu)).\]
\end{enumerate}
\end{lemma}
\begin{proof}
Let $\beta_{\lambda}:=(b_1,b_2, \dots, b_m) \in \mathscr{B}(\lambda)$. We assume that $b_i >0$ for all $i$, which means that the entries of $\beta_{\lambda}$ are the hook lengths of the first column of $[\lambda]$. After reordering we can assume that there are $h_1, h_2 \in \mathbb Z \backslash \{  0 \}$ such that \[\beta_{D(\mu)}:=(b_1-h_1,b_2-h_2, \dots, b_m) \in \mathscr{B}(D(\mu)).\] Assume without loss of generality that $b_1>b_2$ and $b_{1}-h_1>b_{2}-h_2$.

An upper bound for $|(\mu)_{\sim}|$ is the number of indices $i$ such that \[d_m(\beta_{\lambda}, \beta_{D(\mu)}+2^r e_{i})=2.\]
 
There are at most four possibilities for $i$: Either $i=1$, or $i=2$, or if there is an index $j$ (resp. $k$) such that $b_{j}=b_{1}-2^r$ (resp. $b_{k}=b_{2}-2^r$), then $i$ could also be equal to $j$ or $k$. Thus, $|(\mu)_{\sim}|$ is at most $4$.

For the next claim, first observe that $h_2>0$. Indeed, $h_1+h_2=2^r$, so $h_2<0$ implies that $h_1>2^r$. Then \[(b_1-|h_2|,b_2+|h_2|, \dots, b_m) \in \mathscr{B}(\mu).\] But since $b_1>b_2$, it would follow that $\mu \trianglelefteq \lambda$ which is absurd.

Next, we regard the sequence \[\alpha:=\beta_{D(\mu)}+2^r e_1=(b_1-h_1+2^r,b_2-h_2, \dots, b_m).\] If $\alpha \in \mathscr{B}_n$, then by setting $\kappa:=\mathrm{Part}(\alpha)$ we have found a distinguished representative $\kappa \in (\mu)_{\sim}$.

So assume that $\alpha \not\in \mathscr{B}_n $. This means that the entry $b_1-h_1+2^r$ appears twice in $\alpha$, in position $1$ and in another position $i$. Since $b_1-h_1>b_2-h_2$, we know that $i \neq 2$. After a reordering we can assume that $i=3$, i.e., \[\beta_{\lambda}=(b_1,b_2,b_1-h_1+2^r, \dots).\]

Consider the sequence \[\gamma:= (b_2-h_2+2^r,b_1-h_1,b_1-h_1+2^r, \dots).\]

First, since $b_3=b_1-h_1+2^r<2^{r+1}$, it holds that $b_1-h_1<2^r$. Clearly $b_2-h_2 \geq 0$, it follows that \[b_2-h_2+2^r>b_1-h_1.\]

Assume now there is an index $j \geq 4$ such that $b_j$ appears in the first or second entry of $\gamma$. Since $\beta_{D(\mu)}$ is a sequence of $\beta$-numbers, it follows that $b_j \neq b_1-h_1$. So then $b_j=b_2-h_2+2^r$. But then both $\beta_{\lambda}-2^r e_3$ and $\beta_{\lambda}-2^r e_j$ are sequences of $\beta$-numbers which contradicts Corollary \ref{OddnessBetaNumbers}.

We conclude that $\gamma \in \mathscr{B}_n$. Finally, we can now set $\kappa:=\mathrm{Part}(\gamma)$ to be a distinguished representative of $(\mu)_{\sim}.$
\end{proof}

\begin{example}
\label{Example211111511}
Let $\lambda=(2,1^{(5)})$ as in  the examples \ref{Example21111331} and \ref{Example211111EquivalenceClasses}. Let $\mu_1:=(7)$ and $\mu_2:=(5,1,1)$. Then $d(\lambda,D(\mu_1))=d(\lambda, D(\mu_2))=2$ with the size of their respective equivalence classes being $3$ and $1$. 

Let $\beta_{\lambda}:=(7,5,4,3,2,1)$ and $\beta_{\mu_1}=(12,0,4,3,2,1)$ be sequences of $\beta$-numbers of $\lambda$ and $\mu_1$, respectively. Then $\beta_{\mu_1}-4e_1=(8,0,4,3,2,1) \in \mathscr{B}((3))$. So we see that $(7)$ is a distinguished representative of its equivalence class. Similarly, we can confirm the same for $\mu_2$. 
\end{example}

The following definition gives us a setup to handle the classification of the distance $2$ cases.
\begin{definition}
\label{DefinitionClassif Distance 2}
Let $\mu$ be an odd branch of $\lambda$ with $d(\lambda, D(\mu))=2$ and assume that $\mu$ is a distinguished representative in its equivalence class. We fix some sequences of $\beta$-numbers
\begin{align*}
\beta_{\lambda}:=&(b_1,b_2, b_3, \dots, b_m) \in \mathscr{B}(\lambda), \\
\beta_{\mu}:=&(b_1',b_2', b_3, \dots, b_m) \in \mathscr{B}(\mu)
\end{align*}
with $b_i>0$ for all $i$, $b_1 >b_2$ and $b_1'>b_2'$.
Define the sequence \[\alpha_{\mu}:=\beta_{\mu}-2^r e_1+2^r e_2=(b_1'-2^r,b_2'+2^r, b_3, \dots, b_m).\]
\end{definition}

\begin{lemma}
\label{Classif Distance 2}
Assume we are in the situation of Definition \ref{DefinitionClassif Distance 2}. Further assume that $h^{2}_{\lambda}(\mu)>2^r$.
\begin{enumerate}[label=(\roman*)]
\item Assume that $\alpha_{\mu} \in \mathscr{B}_n$. Then $|(\mu)_{\sim}|=4$. Moreover, $(\mu)_{\sim}$ is odd. 
\item Assume that $\alpha_{\mu} \not\in \mathscr{B}_n$. Then $|(\mu)_{\sim}|=3$. Moreover, $\mathrm{OddRank}(\lambda)>0$, $D(\mu)$ is an odd branch of $D(\lambda)$ and $(\mu)_{\sim}$ is odd if and only if $\{D(\mu)  \} \subseteq \mathscr{P}_{D(\lambda)}^{\mathrm{odd}}$ is odd.
\end{enumerate}
\end{lemma}
\begin{proof}
First note that both $b_1,b_2>2^r$. It follows that neither \[\beta_{\lambda}-2^r e_{1}=(b_1-2^r,b_2, b_3, \dots, b_m)\] nor \[\beta_{\lambda}-2^r e_{2}=(b_1,b_2-2^r, b_3, \dots, b_m)\] are sequences of $\beta$-numbers.  Indeed, assume $\beta_{\lambda}-2^r e_{i}$ were a sequence of $\beta$-numbers, for $i$ either $1$ or $2$. Since all the $b_j>0$ and $b_i>2^r$, the sequence $\beta_{\lambda}-2^r e_{i}$ corresponds to the hook lengths of the first column of the partition $D(\lambda) \in \mathscr{P}_{n-2^r}$. This is a contradiction, since $b_{3-i}>2^r>n-2^r$.

So there are indices $i,j>2$ with $b_i=b_1-2^r, b_j=b_2-2^r$. In other words, after a reordering, we have that \[ \beta_{\lambda}=(b_1,b_2,b_1-2^r,b_2-2^r, b_5, \dots, b_m).\] We define 
\begin{align*}
    \gamma_1 := & \, \beta_{\lambda}+(h^1_{\lambda}(\mu)-2^r)e_2-(h^1_{\lambda}(\mu)-2^r)e_3= \\
    =& \, (b_1,b_1'-2^r,b_2',b_2-2^r, b_5, \dots, b_m) \in \mathscr{B}_{n}
\end{align*}
and
\begin{align*}
 \gamma_2 := & \, \beta_{\lambda}+   (h^2_{\lambda}(\mu)-2^r)e_1-(h^2_{\lambda}(\mu)-2^r)e_4= \\
 = & \, (b_1'-2^r,b_2,b_1-2^r,b_2', b_5, \dots, b_m)\in \mathscr{B}_n.
\end{align*}

Let $\kappa_1:=\mathrm{Part}(\gamma_1)$, $\kappa_2:=\mathrm{Part}(\gamma_2)$. It is clear that both $\gamma_1-2^r e_1$ and $\gamma_2-2^r e_2$ are sequences of $\beta$-numbers for $D(\mu)$. Thus $\kappa_1, \kappa_2 \in (\mu)_{\sim}$.

If $\alpha_{\mu}$ is a sequence of $\beta$-numbers, then also  $\alpha_{\mu}-2^r e_2$ is a sequence of $\beta$-numbers for $D(\mu)$ and thus $\kappa_3:= \mathrm{Part}(\alpha_{\mu}) \sim \mu$. So \[(\mu)_{\sim}=\{\mu,\kappa_1, \kappa_2, \kappa_3 \}\] with
\[ \prod_{\nu \in  (\mu)_{\sim}}  \frac{h_{\lambda}^1(\nu) }{h_{\lambda}^2(\nu)} =\frac{h^1_{\lambda}(\mu)}{h^2_{\lambda}(\mu)}  \cdot \frac{h^2_{\lambda}(\mu)}{h^1_{\lambda}(\mu)-2^r} \cdot \frac{h^1_{\lambda}(\mu)}{h^2_{\lambda}(\mu)-2^r} \cdot \frac{h^1_{\lambda}(\mu)-2^r}{h^2_{\lambda}(\mu)-2^r}=\left( \frac{h^1_{\lambda}(\mu)}{h^2_{\lambda}(\mu)-2^r}  \right)^2\] and so  $(\mu)_{\sim}$ is odd.

If on the other hand $\alpha_{\mu}$ is not a sequence of $\beta$-numbers, then the entry $b_2'+2^r$ already appears somewhere in $\beta_{\lambda}$. So after reordering, $\beta_{\lambda}$ is of the form \[(b_1,b_2,b_1-2^r,b_2-2^r, b_2'+2^r, \dots, b_m).\] Thus $\mathrm{OddRank}(\lambda)>0$ and \[\beta_{\lambda}-2^r e_5=(b_1,b_2,b_1-2^r,b_2-2^r, b_2', \dots, b_m) \in  \mathscr{B}(D(\lambda)).\] Moreover, \[(\mu)_{\sim} =\{\mu,\kappa_1,\kappa_2 \}\] and it is clear that $D(\mu) \in \mathscr{P}_{\lambda}^{\mathrm{odd}}$ with $h^{k}_{D(\lambda)}(D(\mu))=h^k_{\lambda}(\mu)-2^r$ for $k=1,2$ from which the statement follows.
\end{proof}

\begin{example}
\begin{enumerate}[label=(\roman*)]
\item Let $\lambda=(4,4,2,2,1,1,1)$ and let $\mu=(13,1,1)$. Then $\mu$ is an odd brach of $\lambda$  with $d(\lambda, D(\mu))=2$ and $\mu$ is a distinguished representative. Moreover, \[(\mu)_{\sim}=\{\mu, (5,5,2,2,1), (5,4,2,2,1,1), (5,3,2,2,1,1,1)  \}\] and we are in the situation of Lemma \ref{Classif Distance 2}(i).
\item It is easy to confirm that the partitions $\lambda=(2, 1^{(5)})$ and $\mu=(7)$ as in the previous examples are
an instance of Lemma \ref{Classif Distance 2}(ii).
\end{enumerate}
\end{example}

\begin{lemma}
\label{Classif Distance 2 Part 2}
Assume we are in the situation of Definition \ref{DefinitionClassif Distance 2}. Further assume that $h^1_{\lambda}(\mu)<2^r$.
\begin{enumerate}[label=(\roman*)]
\item Assume that $\alpha_{\mu} \in \mathscr{B}_n$. Then $|(\mu)_{\sim}|=2$. Moreover, $(\mu)_{\sim}$ is odd. 
\item Assume that $\alpha_{\mu} \not\in \mathscr{B}_n$. Then $|(\mu)_{\sim}|=1$. Moreover, $\mathrm{OddRank}(\lambda)>0$, $D(\mu)$ is an odd branch of $D(\lambda)$ and $(\mu)_{\sim}$ is odd if and only if $\{D(\mu)  \} \subseteq \mathscr{P}_{D(\lambda)}^{\mathrm{odd}}$ is odd.
\end{enumerate}
\end{lemma}
\begin{proof}
One sees easily that the assumption implies that $|(\mu)_{\sim}| \leq 2$. Thus, if $\alpha_{\mu} \in \mathscr{B}_n$ then \[(\mu)_{\sim}=\{\mu, \mathrm{Part}(\alpha_{\mu}) \}\] with \[ \prod_{\nu \in  (\mu)_{\sim}}  \frac{h_{\lambda}^1(\nu) }{h_{\lambda}^2(\nu)} =\frac{h_{\lambda}^1(\mu)}{h_{\lambda}^2(\mu)} \cdot \frac{2^r-h_{\lambda}^2(\mu)}{2^r-h_{\lambda}^1(\mu)}\] and thus $(\mu)_{\sim}$ is odd.
  
So assume now that $\alpha_{\mu} \not\in \mathscr{B}_n$. We can now argue as in Lemma \ref{Classif Distance 2} and see that $\mathrm{OddRank}(\lambda)>0$ and $D(\mu) \in \mathscr{P}_{D(\lambda)}^{\mathrm{odd}}$. Moreover, $(\mu)_{\sim}=\{ \mu \}$ and \[h^{i}_{D(\lambda)}(D(\mu))=2^r-h_{\lambda}^{3-i}(\mu)\] for $i=1,2$, which shows the claim.
\end{proof}

\begin{example}
\begin{enumerate}[label=(\roman*)]
\item Let $\lambda=(6,3,2,1), \mu_1=(9,1,1,1)$ and $\mu_2=(7,2,2,1)$ be partitions of $12$. Then $\mu_1$ and $\mu_2$ are odd branches of $\lambda$ with $d(\lambda, D(\mu_1))=2$ and $(\mu_1)_{\sim} = \{ \mu_1, \mu_2 \}$. Both $\mu_1$ and $\mu_2$ are distinguished representatives in their equivalence class and we are in the situation of Lemma \ref{Classif Distance 2 Part 2}(i).
\item Let $\lambda=(5,5,2,1)$ and let $\mu=(7,6)$ be partitions of $13$. Then $\mu$ is an odd branch of $\lambda$ with $d(\lambda, D(\mu))=2$ and $\mu$ is a distinguished representative. Moreover, \[(\mu)_{\sim}=\{\mu\}\] and we are in the situation of Lemma \ref{Classif Distance 2 Part 2}(ii).
\end{enumerate}
\end{example}

\begin{lemma}
\label{Classif Distance 2 Part 3}
Assume we are in the situation of Definition \ref{DefinitionClassif Distance 2}. Further assume that $h_{\lambda}^{1}(\mu)>2^r$, $h_{\lambda}^{2}(\mu)<2^r$.
\begin{enumerate}[label=(\roman*)]
\item If there is an index $i\geq 3$ with $b_1-b_i=2^r$, then $|(\mu)_{\sim}|=2.$ Moreover, $(\mu)_{\sim}$ is odd.
\item If there does not exist an index $i$ with $b_1-b_i=2^r$, then $|(\mu)_{\sim}|=1$. Moreover, $\mathrm{OddRank}(\lambda)>0$, $D(\mu)$ is an odd branch of $D(\lambda)$ and $(\mu)_{\sim}$ is odd if and only if $\{D(\mu)  \} \subseteq \mathscr{P}_{D(\lambda)}^{\mathrm{odd}}$ is odd.
\item If $b_1-b_2=2^r$, then $|(\mu)_{\sim}|=1$. Moreover, $(\mu)_{\sim}$ is odd.
\end{enumerate}
\end{lemma}
\begin{proof}
Unlike in the lemmas \ref{Classif Distance 2} and \ref{Classif Distance 2 Part 2}, there is no partition in $\mathscr{P}_{\lambda}^{\mathrm{odd}}$ that can come from $\alpha_{\mu}$. Assume that $\alpha_{\mu} \in \mathscr{B}_n$. Then $b_1 > b_1+h_{\lambda}^{2}(\mu)-2^r=b_1'-2^r$ since $h_{\lambda}^{2}(\mu)<2^r$ and $b_1 > b_1-h_{\lambda}^{1}(\mu)+2^r=b_2'+2^r$ since $h_{\lambda}^{1}(\mu)>2^r$. It thus follows that $\mathrm{Part}(\alpha_{\mu}) \trianglelefteq \lambda$, so indeed, $\mathrm{Part}(\alpha_{\mu})$ is not a branch of $\lambda$.

Next, we convince ourselves that $b_2<2^r$. Assume by contradiction that $b_2 \geq 2^r$, so $b_2+h_{\lambda}^{1}(\mu)=b_1'>2^{r+1}$. This means that $b_1'$ can not be the hook length of any $c \in [\lambda]$, in particular, there has to be an entry in $\beta_{\mu}$ that is equal to $0$. Since all the entries $b_i>0$, it follows that $b_2'=0$. But then $b_2=h_{\lambda}^{2}(\mu)<2^r$ by assumption, which is a contradiction. So $|(\mu)_{\sim}| \leq 2$ now follows by the discussion in the proof of Lemma \ref{Distance 2}(i).

Let us now discuss the case of $|(\mu)_{\sim}|=2.$ For this to occur, there has to be an index $i \geq 3$ such that $b_1-b_i=2^r$. Assume this is the case; after some reordering we can assume that $i=3$. So \[\beta_{\lambda}=(b_1, b_2, b_1-2^r, b_4, \dots, b_m).\] Let 
\begin{align*}
 \gamma := & \, \beta_{\lambda}+h^2_{\lambda}(\mu)e_3-h^2_{\lambda}(\mu)e_2= \\
 = & \, (b_1,b_2',b_1'-2^r,b_4, \dots, b_m)\in \mathscr{B}_n.
\end{align*}
Let $\kappa:=\mathrm{Part}(\gamma)$. It is easy to see that $\kappa \in \mathscr{P}_{\lambda}^{\mathrm{odd}}$ and that $(\mu)_{\sim}=\{\mu, \kappa  \}.$ If $b_2>b_1-2^r$, then \[h^{1}_{\lambda}(\kappa)=h^2_{\lambda}(\mu), \ h^{2}_{\lambda}(\kappa) =h^{1}_{\lambda}(\mu)-2^r.\] If $b_2<b_1-2^r$, then \[h^{1}_{\lambda}(\kappa)=h^{1}_{\lambda}(\mu)-2^r, \ h^{2}_{\lambda}(\kappa) =h^{2}_{\lambda}(\mu).\] In both cases, $(\mu)_{\sim}$ is odd.

So in all other cases, $|(\mu)_{\sim}|=1$. Assume that there is no index $i$ such that $b_1-b_i=2^r$. Then $\mathrm{OddRank}(\lambda)>0$ and \[ (b_1-2^r,b_2, \dots, b_m) \in \mathscr{B}(D(\lambda)).\] Moreover, we see that $D(\mu)$ is an odd branch of $D(\lambda)$. We have the same case distinction as before: If $b_2>b_1-2^r$, then \[h^{1}_{D(\lambda)}(D(\mu)) =h^{2}_{\lambda}(\mu), \ h^{2}_{D(\lambda)}(D(\mu))=h^{1}_{\lambda}(\mu)-2^r.\] If $b_2<b_1-2^r$, then \[h^{1}_{D(\lambda)}(D(\mu))= h^{1}_{\lambda}(\mu)-2^r, \ h^{2}_{D(\lambda)}(D(\mu))=h^{2}_{\lambda}(\mu).\]

To finish the proof, assume now that $b_1-b_2=2^r$. Then $h^{1}_{\lambda}(\mu)=h^{2}_{\lambda}(\mu)+2^r$ and so $(\mu)_{\sim}$ is odd. The statement now follows.
\end{proof}

\begin{example} 
\begin{enumerate}[label=(\roman*)]
\item Let $\lambda=(4, 2^{(4)},1^{(3)})$ and let $\mu=(5, 2^{(4)},1,1)$. Then $\mu$ is an odd branch of $\lambda$ with $d(\lambda, D(\mu))=2$ and $\mu$ is a distinguished representative. Moreover, \[(\mu)_{\sim}=\{\mu, (4, 2^{(5)},1)  \}\] and we are in the situation of Lemma \ref{Classif Distance 2 Part 3}(i).
\item Let $\lambda=(10,2)$ and let $\mu=(11,1)$. Then $\mu$ is an odd branch of $\lambda$ with $d(\lambda, D(\mu))=2$ and $\mu$ is a distinguished representative. Moreover, \[(\mu)_{\sim}=\{\mu\}\] and we are in the situation of Lemma \ref{Classif Distance 2 Part 3}(ii). 
\item Let $\lambda=(2, 1^{(5)})$ and let $\mu=(5,1,1)$, as seen in Example \ref{Example211111511}. Then $\mu$ is an odd branch of $\lambda$ with $d(\lambda, D(\mu))=2$ and $\mu$ is a distinguished representative. Moreover, \[(\mu)_{\sim}=\{\mu\}\] and we are in the situation of Lemma \ref{Classif Distance 2 Part 3}(iii). 
\end{enumerate}
\end{example}

We gather the previous results in the following statement:
\begin{corollary}
\label{MainCorollary}
Let $\mu$ be an odd branch of $\lambda$. Then at least one of the following is true:
\begin{enumerate}[label=(\roman*)]
\item $(\mu)_{\sim}$ is odd.
\item $\mathrm{OddRank}(\lambda)>0$, $D(\mu)$ is an odd branch of $D(\lambda)$ and $(\mu)_{\sim}$ is odd if and only if $\{D(\mu)  \} \subseteq \mathscr{P}_{D(\lambda)}^{\mathrm{odd}}$ is odd.
\end{enumerate}
\end{corollary}

We need one final piece to finish the puzzle: 
\begin{lemma}
\label{Preimage of D}
Assume $\mathrm{OddRank}(\lambda)>0$ and let $\kappa$ be an odd branch of $D(\lambda)$. Then there is an odd branch $\mu$ of $\lambda$ such that $D(\mu)=\kappa$ and $(\mu)_{\sim}$ is odd if and only if $\{\kappa\} \subseteq \mathscr{P}_{D(\lambda)}^{\mathrm{odd}}$ is odd.
\end{lemma}
\begin{proof}
Let \[\beta_{\lambda}:=(b_1,b_2,b_3, \dots, b_m) \in \mathscr{B}(\lambda)\] and assume that, after a suitable reordering, \[(b_1-2^r,b_2,b_3, \dots, b_m) \in \mathscr{B}(D(\lambda)).\] There are two indices $i,j$ and a positive integer $h$ such that \[\beta_{\lambda}-2^r e_1+h e_i-h e_j \in \mathscr{B}(\kappa)\] and $b_i>b_j$. There are three possibilities: Either $1 \notin \{i,j\}$, or $i=1$, or $j=1$. We will regard \[\alpha:=\beta_{\lambda}+h e_i-h e_j.\]

Assume first that $1 \notin \{i,j\}$. After a reordering we can assume that $i=2, j=3$. There are three further possibilities: Either $\alpha \in \mathscr{B}_n$, or $b_1=b_3-h$, or $b_1=b_2+h$.

If $\alpha \in \mathscr{B}_n$, then it is clear that we can set $\mu:=\mathrm{Part}(\alpha)$ with $\mu \in \mathscr{P}_{\lambda}^{\mathrm{odd}}$ and that $d(\lambda, \kappa)=3$. We are thus in the situation of Lemma \ref{Distance 3} and we are done.

Assume now that $b_1=b_3-h$. Then \[\gamma:=(b_1-2^r,b_2+h+2^r,b_3-h,b_4, \dots, b_m) \in \mathscr{B}_n.\] If we now set $\mu:=\mathrm{Part}(\gamma)$, then clearly $\mu \in \mathscr{P}_{\lambda}^{\mathrm{odd}}$ and $d(\lambda, \kappa)=2$. Moreover, $\mu$ is a distinguished representative of its equivalence class and \[h^{1}_{\lambda}(\mu)=b_2-b_3+h+2^r, \ h^{2}_{\lambda}(\mu)=h+2^r,\] so we are in the case of Lemma \ref{Classif Distance 2}(ii).

Assume now that $b_1=b_2+h$. Here, we can set \[\gamma:=(b_1-2^r,b_2+h,b_3-h+2^r,b_4, \dots, b_m) \in \mathscr{B}_n.\] We again set $\mu:=\mathrm{Part}(\gamma)$ and again, $d(\lambda, \kappa)=2$ and $\mu$ is a distinguished representative of its equivalence class. We calculate that \[h^{1}_{\lambda}(\mu)=2^r-h, \ h^{2}_{\lambda}(\mu)=2^r-(b_2-b_3+h),\] so we are in the case of Lemma \ref{Classif Distance 2 Part 2}(ii). 

Being done with that case, we can now assume that either $i=1$ or $j=1$. After reordering, we have \[(b_1-2^r+\delta h,b_2- \delta h, b_3, \dots, b_m) \in \mathscr{B}(\kappa)\] for 
\[ \delta=
\begin{cases}
1, & \text{if} \ i=1, \\
-1, & \text{if} \ j=1.
\end{cases}
\]
We set \[\gamma:=(b_1-2^r+\delta h,b_2- \delta h+2^r,b_3, \dots, b_m) \in \mathscr{B}_n \]  and let $\mu:=\mathrm{Part}(\gamma)$ be the corresponding partition. It is easy to see that $d(\lambda, \kappa)=2$ and $\mu$ is a distinguished representative of its equivalence class. We quickly confirm that we are in the case of Lemma \ref{Classif Distance 2 Part 3}(iii). This concludes the proof.
\end{proof}

We can now finally prove the main theorem:

\begin{proof}(of Theorem \ref{Main-Thm-Symmetric})
We will show that $\mathscr{P}_{\lambda}^{\mathrm{odd}}$ is odd by induction on the oddness rank of $\lambda$.
To begin, the statement for $\mathrm{OddRank}(\lambda)=0$ immediately follows by Corollary \ref{MainCorollary}.

So assume now $\mathrm{OddRank}(\lambda)>0$ and that we know the result to hold for $D(\lambda)$. We define the set \[M := D^{-1}( \mathscr{P}_{D(\lambda)}^{\mathrm{odd}}) \subseteq \mathscr{P}_{\lambda}^{\mathrm{odd}},\] where the subset relation holds by Lemma \ref{OddnessPartitions}. 

We decompose \[\mathscr{P}_{\lambda}^{\mathrm{odd}}=M \cup M'\] for some other set $M' \subseteq \mathscr{P}_{\lambda}^{\mathrm{odd}}$ into a disjoint union. Lemma \ref{Preimage of D} tells us that $M$ is odd if and only if $\mathscr{P}_{D(\lambda)}^{\mathrm{odd}}$ is odd, which we know by the induction hypothesis. Moreover, for all $ \mu \in M'$, it holds that $D(\mu)$ is not an odd branch of $D(\lambda)$ and so by Corollary \ref{MainCorollary}, the set $(\mu)_{\sim}$ is odd; it follows that the whole set $M'$ must be odd, and so $\mathscr{P}_{\lambda}^{\mathrm{odd}}$ is odd. The theorem now follows by Remark \ref{RemarkEquivalentThm}.
\end{proof}

\begin{remark}
We have confirmed that $a_{\lambda}^{(2)}$ is even for any even partition $\lambda$. So if there is a partition $\mu$ such that $a_{\mu}^{(2)}$ is odd, then $\mu$ must necessarily be odd as well. The reason this can occur for odd partitions relies critically in the failure of Lemma \ref{Distance 1}.  We denote \[A(n):= \# \{ \mu \in \mathscr{P}_n \mid a_{\mu} \ \text{is odd}  \}.\] and \[B(n):= \# \{ \mu \in \mathscr{P}_n \mid \mu \ \text{is odd}  \}.\]

It is clear that $B(n)$ is an upper bound for $A(n)$; an explicit formula for $B(n)$ can be found in \cite{MacDonald}. With some code in GAP \cite{GAP4}, we calculated the values of $A(n)$ for $1 \leq n \leq 100$. 

\begin{table}[ht]
    \centering
    \begin{tabular}{c||c |c |c |c |c |c |c |c |c |c}
       $n$ & $1$ & $2$ & $3$ & $4$ & $5$ & $6$ & $7$ & $8$& $9$ & $10$ \\ \hline  \addlinespace[0.1em]
    $A(n)$ & $0$ & $1$ & $1$ & $2$ & $2$ & $4$ & $4$ & $4$ & $4$ & $4$ \\ \hline  \addlinespace[0.1em]
    $B(n)$ & $1$ & $2$ & $2$ & $4$ & $4$ & $8$ & $8$ & $8$ & $8$ & $16$ \\
    \multicolumn{11}{c}{} \\
      $n$ & $11$ & $12$ & $13$ & $ 14$ & $15$ & $16$ &  $17$ & $18$ & $19$ & $20$\\ \hline  \addlinespace[0.1em]
    $A(n)$ & $8$ & $20$ & $16$ & $32$ & $32$ & $8$ & $8$ & $16$ & $16$ & $24$  \\ \hline  \addlinespace[0.1em]
    $B(n)$ & $16$ & $32$ & $32$ & $64$ & $64$ & $16$ & $16$ & $32$ & $32$ & $64$ \\
     \multicolumn{11}{c}{} \\
      $n$ &  $21$ & $22$ & $23$ & $24$& $25$ & $26$ & $27$ & $28$ & $29$ & $30$ \\ \hline  \addlinespace[0.1em]
    $A(n)$ & $24$ & $64$ & $64$ & $64$ & $64$ & $128$ & $128$ & $256$ & $256$ & $512$  \\ \hline  \addlinespace[0.1em]
    $B(n)$ & $64$ & $128$ & $128$ & $128$ & $128$ & $256$ & $256$ & $512$ & $512$ & $1024$  \\
    \multicolumn{11}{c}{} \\ 
     $n$ &  $31$ & $32$ & $33$ & $34$ & $35$ & $36$ & $37$ & $38$& $39$ & $40$\\ \hline  \addlinespace[0.1em]
    $A(n)$ & $480$ & $16$ & $16$ & $28$ & $32$ & $68$ & $64$ & $128$ & $128$ & $112$ \\ \hline  \addlinespace[0.1em]
    $B(n)$ & $1024$ & $32$ & $32$ & $64$ & $64$ & $128$ & $128$ & $256$ & $256$ & $256$ 
    \end{tabular}
      \caption{The values of $A(n)$ and $B(n)$ for the first $40$ values of $n$.}
        \label{Table1}
\end{table}
Some patterns are easily observed: It seems that $A(n)=B(n)/2$ in most cases; for $1 \leq n \leq 100$, this holds true for $75$ values of $n$. Moreover, it holds that the difference $B(n)/2-A(n) \geq 0$ with only a few exceptions; these being $n=12$, $36$, $48$, $54$, $91$ and $92$ for $n \leq 100$. The values of $|B(n)/2-A(n)|$ seem to be either an integer power of $2$, or a multiple of $12$; the nonzero multiples of $12$ occurring for $n=41$, $54$, $61$, $90$ and $93$ for $n \leq 100$. 
\end{remark}

\begin{Ack}
I want to thank Gabriele Nebe and Jonas Hetz for their feedback on parts of this paper. This paper is a contribution to Project-ID 286237555 – TRR 195 – by the
Deutsche Forschungsgemeinschaft (DFG, German Research Foundation).
\end{Ack}

\bibliography{References}
\end{document}